\documentclass[12pt]{amsart}
\usepackage[utf8]{inputenc}
\usepackage{amsmath, amssymb, amsthm}
\usepackage[margin=1in]{geometry}
\usepackage{xcolor}
\usepackage{enumerate}
\definecolor{purple}{rgb}{.5,0,1}

\usepackage{cite}
\usepackage[
	colorlinks,
	linkcolor={black},
	citecolor={black},
	urlcolor={black}
]{hyperref}

\newtheorem{theorem}{Theorem}[section]
\theoremstyle{definition}
\newtheorem{definition}[theorem]{Definition}
\newtheorem{example}[theorem]{Example}

\newtheorem{proposition}[theorem]{Proposition}
\newtheorem{corollary}[theorem]{Corollary}
\newtheorem{lemma}[theorem]{Lemma}
\newtheorem{remark}[theorem]{Remark}
\newtheorem*{question}{Question}

\newcommand{\bbN}{\mathbb{N}}
\newcommand{\bbR}{\mathbb{R}}
\newcommand{\bbZ}{{\mathbb{Z}}}

\newcommand{\rmD}{{\mathrm{D}}}

\newcommand{\tilF}{{\widetilde{F}}}
\newcommand{\tilK}{{\widetilde{K}}}

\newcommand{\dist}{{\mathrm{dist}}}
\newcommand{\diam}{{\mathrm{diam}}}
\newcommand{\RV}{{\Lambda}}
\newcommand{\BRV}{{\mathrm{BRV}}}
\newcommand{\TRV}{{\mathrm{TRV}}}
\newcommand{\set}[1]{\left\{{#1}\right\}}

\numberwithin{equation}{section}

\title{On Sums of Semibounded Cantor Sets}
\author[J.\ Fillman]{Jake Fillman}
\author[S.\ H.\ Tidwell]{Sara H.\ Tidwell}
\date{}

\thanks{J.F.\ was supported in part by Simons Foundation Collaboration Grant \#711663. J.F.\ and S.H.T.\ were supported in part by undergraduate research funding from the Texas State University Research Enhancement Program.}
\begin{document}
\maketitle

\begin{abstract}
Motivated by questions arising in the study of the spectral theory of models of aperiodic order, we investigate sums of functions of semibounded closed subsets of the real line. We show that under suitable thickness assumptions on the sets and growth assumptions on the functions, the sums of such sets contain half-lines. We also give examples to show our criteria are sharp in suitable regimes.
\end{abstract}

\setcounter{tocdepth}{1}
\tableofcontents
\hypersetup{
	linkcolor={black!30!blue},
	citecolor={red},
	urlcolor={black!30!blue}
}

\section{Introduction}
\subsection{Background and Motivation}
The present paper is concerned with the following question:
\begin{question}
If $F \subseteq \bbR$ is closed and bounded from below and $g:[\inf(F),\infty) \to \bbR$ is increasing and continuous, under what conditions (on $F$, $g$, or both) does $g[F] + g[F]$ contain a half-line?
\end{question}

This question has its roots in some recent work on spectral theory of multidimensional quasicrystals. To set the stage, we will explain how this question arose and then give an answer: If $F$ is thick  and the relative Lipschitz behavior of $g$ is controllable (in senses to be made precise later), then $g[F] + g[F]$ contains a half-line. Moreover, we will show that these hypotheses are necessary by exhibiting examples in which they fail and the resultant sums do not contain half-lines.

Since their discovery in the 1980s by Shechtman \textit{et al.}\ \cite{SBGC1984PRL}, quasicrystals\footnote{That is, mathematical or physical structures simultaneously exhibiting aperiodicity and long-range order.} have been studied intensively by mathematicians and physicists. We direct the readers to the books \cite{BaaGri2013Vol1, BaaMoo2000} and the references therein for background. This paper is motivated by questions that arise when one studies the spectra of multidimensional quasicrystal models. In particular, one asks whether the spectra of multidimensional continuum quasicrystal models must contain a half-line. This question itself is motivated by the corresponding question for crystalline (i.e., periodic) models, called the \emph{Bethe--Sommerfeld conjecture}. The conjecture is now a theorem with progress by many authors \cite{HelMoh98, Karp97, PopSkr81, Skr79, Skr84, Skr85, Vel88}, with a full resolution by Parnovski \cite{Parn2008AHP}.

To study electronic properties of quasicrystals, one often considers a single-particle Hamiltonian in $L^2(\bbR^d)$ of the form
\[ L_V = -\Delta + V\]
in which the potential $V:\bbR \to \bbR$ is pattern-equivariant with respect to a suitable model of aperiodic order (e.g., the Fibonacci or Penrose tilings).
 Such models have been heavily studied in dimension $d=1$, where there is a cornucopia of tools available to study the relevant spectral theory, and as a consequence, one has many results about the spectrum, spectral type, and density of states; see, e.g. \cite{Bel1990, BelBovGhe1991, BelIocScoTes1989CMP, DamanikGorodetski2011CMP, DamGor2012GAFA, DamGorYes2016Invent, LiuQu2015CMP, LiuQuYao2017JSP, Suto1987CMP, Suto1989JSP}, for a small sample of the literature. There are comparatively fewer spectral results in higher dimensions, partially due to the disappearance of some of the key tools in the transition from dimension $d=1$ to dimensions $d\geq 2$; compare, e.g., \cite{HMTpreprint, Hof1993JSP, Hof1995JSP, KLS2003CMP, LenzStol2003OAAMP, LenzStol2005JAM}. 
 One fruitful method to push results into higher dimensions is to study \emph{separable} potentials, that is, potentials of the form
\[ V(x) = \sum_{j=1}^d V_j(x_j), \]
where each $V_j:\bbR \to \bbR$ is a well-understood one-dimensional model. In particular, for the spectrum, this results in
\[ \sigma(L_V) = \sigma(L_{V_1}) + \sigma(L_{V_2})+ \cdots + \sigma(L_{V_d}) = \set{\sum_{j=1}^d \lambda_j : \lambda_j \in \sigma(L_{V_j})},\]
the Minkowski sum of sets. 

On one hand, the spectra of one-dimensional  quasicrystal models have a strong tendency to be zero-measure Cantor sets \cite{BelBovGhe1991, BelIocScoTes1989CMP, DFG2014AHP, DamLen2006DMJ, LenSeiSto2014JDE, Suto1987CMP, Suto1989JSP}. On the other hand, the self-sums of fractal sets of zero Lebesgue measure can have nonempty interior. For instance, it is well-known that if $K \subseteq [0,1]$ denotes the standard middle-thirds Cantor set, one has $K+K=[0,2]$. In particular, the self-sum of a zero-measure set may be an interval.

In the recent work \cite{DFG2021JFA}, it was shown that if $V_j(x)$ denotes a suitable locally constant version of the Fibonacci potential, then $\sigma(L_V)$ contains a half-line for each $d \geq 2$. One of the key ingredients of the proof was an affirmative answer to the main question for the function $g(x)=x^2$ and Cantor sets sufficiently thick in the sense of Newhouse \cite{Newhouse1970Proc, Newhouse1979PMIHES}. This is how Minkowski sums of unbounded fractal sets arise in spectral theory.

In the present work, we will demonstrate classes of functions and sets for which the result of \cite{DFG2021JFA} holds. We give a class of examples for which one can observe a sharp phase transition between containing and not-containing a half-line.

\subsection{Definitions and Results}
Let us begin by defining terminology and recalling some results that will be useful. 

We first discuss the class of functions with which we work. In the sequel, we will consider closed sets and functions that are bounded from below. Consequently, after shifting, we are free to assume without loss of generality that all sets are contained in $\bbR_+ = [0,\infty)$ and that all functions map $\bbR_+$ to itself.

Loosely speaking, the class of functions for which the result holds true is that of monotonic locally Lipschitz function for which the local Lipschitz constants do not vary too quickly in a relative sense, which we make precise in the following definition.

\begin{definition}
We say that $g: \bbR_+ \to \bbR_+$ is an \emph{admissible function} if it is continuous, strictly increasing, and satisfies $\lim_{x\to\infty}g(x)=\infty$. For $x \in \bbR_+$, let $\rmD^\pm$ denote the upper and lower derivatives:
\[ \rmD^+(g,x) = \limsup_{z \to x} \frac{g(z)-g(x)}{z-x}, \quad \rmD^-(g,x) = \liminf_{z \to x} \frac{g(z)-g(x)}{z-x}. \]
Given an admissible function $g$ and constants $\gamma>0$ and $M\geq 0$, define
\begin{equation}
\RV(g,\gamma,M) = \sup\set{\frac{\rmD^+(g,x)}{\rmD^-(g,y)} : x,y \geq M, \ |x-y|\leq \gamma}, 
\end{equation}
 Note that $\RV(g,\gamma,M) \geq 1$ for all $M$ and $\gamma$ and for fixed $\gamma$ is nonincreasing in $M$ and define
\begin{equation}
    \RV(g,\gamma):= \lim_{M \to\infty} \RV(g,\gamma, M).
\end{equation}
We say that $g$ has \emph{bounded relative variation} (in short: $g \in\BRV$) if $\RV(g,\gamma)<\infty$ for some $\gamma > 0$ and \emph{trivial relative variation} ($g \in\TRV)$ if $\RV(g,\gamma)=1$ for some $\gamma$. We will show later that neither of these notions depends on the particular choice of $\gamma$ (cf.\ Prop.\ \ref{prop:BRVnotgammadep}) and hence one has
\begin{align*}
    \TRV &= \set{g : g \text{ is admissible and } \RV(g,1) =1} \\
    \BRV &= \set{g : g \text{ is admissible and } \RV(g,1) < \infty}.
\end{align*}
\end{definition}
\begin{remark}
Let us make some comments about the definitions.
\begin{enumerate}[{\rm(a)}]
    \item We think it is extremely probable that these spaces of functions have been considered in other works, likely under a different name. However, we were unable to locate a precise reference. 
    \smallskip
    
    \item The assumptions imply that $\BRV$ functions are differentiable almost everywhere and locally Lipschitz continuous for sufficiently large $x$. If $g$ is everywhere differentiable and $g \in \BRV$, then one can check that $\log g'$ is a function of bounded variation on any compact subinterval of $[M,\infty)$ where $M$ is sufficiently large.
    \smallskip
    
    \item However, we find it prudent not to restrict to differentiable functions, first, since the results do not need that assumption, and second, piecewise affine functions supply a useful set of test cases and examples to consider.

\end{enumerate}
\end{remark}

Next, we describe the kinds of closed sets to which we may apply our results. In order to do this, let us formulate precisely what it means to say that a set is \emph{thick}. The ideas and key results date back to work of Newhouse \cite{Newhouse1970Proc, Newhouse1979PMIHES}. Let $K \subseteq \mathbb{R}$ be a compact set and denote by $I = [\min K, \max K]$ its convex hull. Any connected component of $I\backslash K$ is called a \emph{gap} of $K$. A \emph{presentation} of $K$ is given by an ordering $\mathcal{U} = \{U_n\}_{n \ge 1}$ of the gaps of $K$. If $u \in K$ is a boundary point of a gap $U$ of $K$, we denote by $B$ the connected component of $I\backslash (U_1\cup U_2\cup\ldots \cup U_n)$ (with $n$ chosen so that $U_n = U$) that contains $u$ and write
$$
\tau(K, \mathcal{U}, u)=\frac{|B|}{|U|}.
$$
The \emph{thickness} $\tau(K)$ of $K$ is given by
$$
\tau(K) = \sup_{\mathcal{U}} \inf_{u} \tau(K, \mathcal{U}, u).
$$
One can check that $\tau(K) = \infty$ if and only if $K$ is a closed interval. It is well-known that one can take as a maximizer any presentation $\mathcal{U}$ in which the gaps are ordered in such a way that the gap lengths are nonincreasing \cite{PalisTakens1993}.

The following consequence of the Newhouse gap lemma \cite{Newhouse1970Proc, Newhouse1979PMIHES} is stated as \cite[Lemma~6.2]{DamanikGorodetski2011CMP} and proved there.

\begin{lemma}\label{l.sumofcs}
Suppose $K, K'\subseteq \mathbb{R}$ are compact sets with $\tau(K)\cdot \tau(K')>1$. Assume also that the size of the largest gap of $K'$ is not greater than the diameter of $K$, and the size of the largest gap
of $K$ is not greater than the diameter of $K'$. Then,
$$
K + K' = [\min K + \min K', \max K + \max K'].
$$
\end{lemma}

\begin{remark}\label{r.cplusc}
A particular consequence of Lemma~\ref{l.sumofcs} is the following: if $K \subseteq \mathbb{R}$ is a Cantor set with $\tau(K) > 1$, then
$$
K + K = [2 \min K, 2 \max K].
$$
\end{remark}

Thickness is a ubiquitous notion in geometric analysis. For a partial list, it has applications in  geometric measure theory \cite{McDTaylor2021, SimonTaylor2020MPCPS}, number theory \cite{Yu2020}, fractal geometry \cite{Yavicoli2021IJM}, and pinned distance problems \cite{McDTaylor2021}. Multidimensional generalizations have been considered in \cite{SimonTaylor2020MPCPS, Yavicoli2022preprint}. For additional information about thickness, we direct the reader to \cite{Astels2000TAMS, PalisTakens1993}.

As mentioned before, we will consider only semibounded closed sets, so after shifting, we will also always assume that said sets lie in $\bbR_+ = [0,\infty)$.

\begin{definition}
 An \emph{ordered fragmentation} of a semibounded closed set $F$ is a decomposition
\begin{equation}
    F= \bigcup_{n=0}^\infty K_n,
\end{equation}
where each $K_n$ is compact and nonempty, and $\max K_n < \min K_{n+1}$ for all $n$. We call $K_n$ the $n$th \emph{fragment} of $F$. 

Given constants $A,a,\tau>0$, we say that $F$ is $(A,a,\tau)$-\emph{thick} if $F$ has an ordered fragmentation such that
\begin{align}\label{eq:AatThick1}
    A \leq \diam(K_n) \leq 2A \quad \forall n \geq 0, \\
    \label{eq:AatThick2}
    \dist(K_n,K_{n+1}) < a \quad  \forall n \geq 0, \\
    \label{eq:AatThick3}
    \tau(K_n) \geq \tau \quad \forall n \geq 0.
\end{align}
\end{definition}

\begin{theorem}\label{t:AbstracBSC}
Assume $F\subseteq \bbR_+$ is an $(A,a,\tau)$-thick semibounded closed set for constants $A>a>0$ and $\tau >1$.
Suppose $g \in \TRV$, and let $\tilF = g[F]$. Then $\widetilde F+\widetilde F$ contains a half-line.
\end{theorem}

One can weaken the assumption $g \in \TRV$ ($\RV(g,\gamma)=1$) somewhat, depending upon the constants $A,a > 0$. Namely, one can show that $\tilF+ \tilF$ contains a half-line if $F$ is $(A,a,\tau)$ thick for some $\tau>1$ and $\RV(g)$ is sufficiently small.

\begin{theorem} \label{t:AbstracBSCgen}
Given $A>a>0$ and $\varepsilon >0$, there exists $\delta = \delta(A,a,\varepsilon)>0$ such that if $F$ is $(A,a,1+\varepsilon)$-thick and  $g\in \BRV$ has $1 \leq \RV(g,A) < 1+\delta$, then $\widetilde F+\widetilde F$ contains a half-line, where $\tilF=g[F]$.
\end{theorem}

\begin{remark} Let us make a few remarks.
\begin{enumerate}[{\rm(a)}]
\item The proof of Theorem~\ref{t:AbstracBSCgen} gives an explicit bound on $\delta$, but we have not optimized the constants in the proof.
\smallskip

\item One can clearly relax some of the hypotheses. For instance, in Theorem~\ref{t:AbstracBSCgen}, it suffices that $F$ be \textit{eventually} $(A,a,1+\varepsilon)$-thick in the sense that $F \cap [c,\infty)$ is $(A,a,1+\varepsilon)$-thick for some $c>0$. One can similarly relax hypotheses on $g$.
\end{enumerate}
\end{remark}
Let us remark that the assumptions of Theorem~\ref{t:AbstracBSC} are met with $f$ any  polynomial function nonnegative on $\bbR_+$, as well as subexponential functions such as $\exp(x^\alpha)$ $0<\alpha<1$. The assumptions of Theorem~\ref{t:AbstracBSCgen} are met by those functions and exponential functions of the type $\exp(rx)$ with $r>0$ sufficiently small. We give an account of these objects in Section~\ref{sec:BRV} including a discussion of the structure of the spaces $\TRV$ and $\BRV$. One can use this to prove spectral results for separable operators modelled on the Fibonacci tiling. To define such potentials, fix $\lambda > 0$, define
\begin{equation} V_\lambda(x) = \lambda \sum_{n\in \bbZ} \chi_{[1-\alpha,1)}(n\alpha \ \mathrm{mod} \ 1) \chi_{[n,n+1)}(x)  \end{equation}
and put $\Sigma_\lambda = \sigma(-\Delta+V_\lambda)$.

\begin{corollary} \label{coro:H^s}
For any $\lambda>0$, $s>0$, $\Sigma_\lambda^s + \Sigma_\lambda^s$ contains a half-line.
\end{corollary}

\begin{proof}
It was shown in \cite{DFG2021JFA} that $\Sigma_\lambda^{1/2}$ is eventually $(A,a,\tau)$-thick for suitable $A,a,\tau$. Since $x\mapsto x^{2s}$ is in $\TRV$, the result follows immediately from Theorem~\ref{t:AbstracBSC}.
\end{proof}

Of course, the sum in Corollary~\ref{coro:H^s} corresponds to the operator $$H_x^s+H_y^2 = (-\partial_x^2+V_\lambda(x))^s + (-\partial_y^2+V_2(y))^s.$$ 
It would also be of interest to study the operator $(-\Delta)^s+V_\lambda(x)+V_\lambda(y)$, but this is not within reach of the methods of this paper.
\medskip

One can turn the formulation of Theorem~\ref{t:AbstracBSCgen} around and see that the result also holds in an appropriate  dual asymptotic regime of bounded $\Lambda(g,A)$ and sufficiently large thickness. More precisely, Theorem~\ref{t:AbstracBSCgen} can be viewed as fixing an $(A,a,\tau)$-thick set $F$ and proving the desired half-line statement for $g \in \BRV$ with $\Lambda$ sufficiently small. One can also fix a $\BRV$ function and prove a similar statement if $F$ is ``sufficiently thick''.

\begin{theorem} \label{t:BSC:bigtau}
Given $A>0$ and $R > 1$, there exist constants $a_0 = a_0(A,R)$ and $\tau_0 =\tau_0(R)>0$ such that the following holds. If $F \subseteq \bbR_+$ is $(A,a,\tau)$-thick for some $0<a \leq a_0$ and $\tau \geq \tau_0$ and $\Lambda(g,A) \leq R$, then $\tilF + \tilF$ contains a half-line.
\end{theorem}

To round out the discussion, we give an example to show that $\tilF+\tilF$ may not contain a half-line. To that end, let us say that $F$ is $a$-sparse if it has infinitely many open gaps of length at least $a$, that is, if $F$ enjoys an ordered fragmentation $\{K_n\}_{n=0}^\infty$ satisfying
\begin{align}
    \dist(K_n,K_{n+1}) \geq a \quad  \forall n \geq 0.
\end{align}
Notice that the thickness of the pieces is irrelevant. Indeed, the following result remains true even if all fragments are closed intervals.

\begin{theorem} \label{t:main:counterex}
Given $r>0$, let $g_r(x) = e^{rx}$. Given $a>0$, if $F$ is an $a$-sparse closed set, $ra \geq \log 2$, and $\widetilde{F} = g_r[F]$, then $\tilF +\tilF$ does not contain a half-line.
\end{theorem}

We will also show that $ra \geq \log 2$ is sharp in the previous theorem by giving an example for each $a<r^{-1}\log 2$ of an $a$-sparse set with $\tilF + \tilF$ containing a half-line.

Of course, it is no accident that the phase transition (from containing to not-containing a half line) occurs precisely at $e^{ra} = 2$, since one can easily check that $\Lambda(e^{rx},a) = e^{ra}$. 

The remainder of the paper is organized as follows. We prove some basic estimates for $\BRV$ functions in Section~\ref{sec:BRV}. We prove Theorems~\ref{t:AbstracBSC}, \ref{t:AbstracBSCgen}, and \ref{t:BSC:bigtau} in Section~\ref{sec:BSC}, and we prove Theorem~\ref{t:main:counterex} in Section~\ref{sec:nohalfline}. 

\subsection*{Acknowledgements} We are grateful to Evyi Palsson for helpful conversations.  J.F.\ also thanks the American Institute of Mathematics for hospitality and support during a January 2022 visit, during which part of this work was completed.

\section{Functions of Bounded Relative Variation} \label{sec:BRV}

We collect here some useful properties of functions in $\BRV$. We will use the following calculation somewhat frequently. Although it is well-known, we include the proof for the reader's convenience and to keep the paper more self-contained. 

\begin{lemma} \label{lem:poorMansMVT}
For any $g:[a,b]\to\bbR$,
\begin{equation}
    \inf_{x \in [a,b]} \rmD^-(g,x) 
    \leq \frac{g(b) - g(a)}{b-a}
    \leq \sup_{x \in [a,b]} \rmD^+(g,x).
\end{equation}
\end{lemma}

\begin{proof}
First, observe that if $r<t$, $\lambda \in (0,1)$ and $s = \lambda r + (1-\lambda)t$ then
\begin{equation} \label{eq:slope:weightedavg}
\frac{g(t)-g(r)}{t-r}
=\lambda \frac{g(t)-g(s)}{t-s}
+(1-\lambda)\frac{g(s)-g(r)}{s-r} .
\end{equation}
Using \eqref{eq:slope:weightedavg}, start with $[a_0,b_0]=[a,b]$ and choose $[a_0,b_0] \supseteq [a_1,b_1]\supseteq \cdots$ so that $[a_{n+1},b_{n+1}]$ has half the length of $[a_n,b_n]$ and
\[\frac{g(b_{n+1}) - g(a_{n+1})}{b_{n+1}-a_{n+1}} \geq \frac{g(b_n) - g(a_n)}{b_n - a_n}.\]
Choosing $x^*$ in the intersection of all $[a_n,b_n]$, one can use \eqref{eq:slope:weightedavg} again to see that
\begin{equation}
    \sup_{x \in [a,b]}\rmD^+(g,x) 
    \geq \rmD^+(g,x^*) 
    \geq \frac{g(b)-g(a)}{b-a}.
\end{equation}
The other inequality is proved in the same manner.
\end{proof}

Using Lemma~\ref{lem:poorMansMVT}, we deduce the following inequality that relates average rates of change for admissible functions to $\Lambda$.

\begin{lemma} \label{lem:gratioToxRatioViaRV}
Suppose $g$ is an admissible function and $M \geq 0$. For any intervals $[a,b],[c,d] \subseteq [M,\infty)$ for which $\diam([a,b]\cup[c,d])\leq \gamma$, one has
\begin{equation}
    \frac{g(d)-g(c)}{g(b)-g(a)} \geq [\RV(g,\gamma,M)]^{-1} \frac{d-c}{b-a}.
\end{equation}
\end{lemma}

\begin{proof}
By the previous lemma,
\begin{equation}
    \frac{g(d) - g(c)}{d-c} \geq \inf_{z \in [c,d]} \rmD^-(g,z), \quad 
    \frac{g(b) - g(a)}{b-a} \leq \sup_{w \in [a,b]} \rmD^+(g,w).
\end{equation}
Since $g$ is increasing on $[M,\infty)$, the previous inequalities yield
\begin{equation}
    \frac{g(d)-g(c)}{g(b)-g(a)} \geq \left[\frac{\inf_{z \in [c,d]} \rmD^-(g,z)}{\sup_{w \in [a,b]} \rmD^+(g,w)}\right] \frac{d-c}{b-a}.
\end{equation}
The result follows by definition of $\RV(g,\gamma,M)$.
\end{proof}

We begin by clarifying some of the main properties of $\Lambda(g,\gamma)$. First, it is submultiplicative in the sense that $\Lambda(g,\gamma_1+\gamma_2) \leq \Lambda(g,\gamma_1)\Lambda(g,\gamma_2)$.

\begin{proposition} \label{prop:BRV:gamma12}
Suppose $g:\bbR_+ \to \bbR_+$ is admissible. For any $\gamma_1,\gamma_2>0$,
\begin{align}
 \label{eq:BRV:gamma1+gamma2}
    \RV(g,\gamma_1+\gamma_2) 
     \leq \RV(g,\gamma_1)\RV(g,\gamma_2).
\end{align}
\end{proposition}

\begin{proof}
Let $M \geq 0$. If $x,y \geq M$ and $|x-y| \leq \gamma_1+\gamma_2$, we choose $z \geq M$ for which $|x-z| \leq \gamma_1$ and $|z-y|\leq \gamma_2$, leading to
\[ \frac{\rmD^+(g,x)}{\rmD^-(g,y)} 
\leq  \frac{\rmD^+(g,x)}{\rmD^-(g,z)}   \frac{\rmD^+(g,z)}{\rmD^-(g,y)} 
\leq \RV(g,\gamma_1,M) \RV(g,\gamma_2,M),  \]
from which \eqref{eq:BRV:gamma1+gamma2} follows by sending 
$M\to\infty$.
\end{proof}

The previous calculation enables us to see that the sets $\TRV$ and $\BRV$ do not depend on the choice of $\gamma$ used to test $\Lambda(g,\gamma)$. More precisely,

\begin{proposition} \label{prop:BRVnotgammadep}
Suppose $g: \bbR_+ \to \bbR_+$ is admissible. 
\begin{enumerate}[{\rm(a)}]
    \item $g \in \TRV$ if and only if $\RV(g,\gamma)=1$ for all $\gamma>0$.
    \item $g \in \BRV$ if and only if $\RV(g,\gamma) < \infty$ for all $\gamma>0$.
\end{enumerate}
\end{proposition}

\begin{proof}(a)
One direction is trivial. For the other direction, suppose $g \in \TRV$, which implies $\Lambda(g,\gamma')=1$ for some $\gamma'>0$. By Proposition~\ref{prop:BRV:gamma12}, we get $\Lambda(g,n\gamma')\leq 1$ for all $n$. Since $\Lambda(g,\gamma) \geq 1$ for all $\gamma$ and $\Lambda(g,\cdot)$ is nondecreasing  $\Lambda(g,\gamma) =  1$ for every $\gamma$. 

The proof of (b) is almost identical.\end{proof}

Next, we discuss the arithmetic properties of $\TRV$ and $\BRV$. We will show that these sets are closed under sums and products. The following bound supplies the needed input to prove that both sets are closed under sums.

\begin{proposition} \label{prop:BRV:g+h}
Suppose $g,h:\bbR_+ \to \bbR_+$ are admissible.  For all $\gamma > 0$, one has
\begin{align}
\label{eq:BRV:M(g+h)2}
    \RV(g+h,\gamma) &\leq \max( \RV(g,\gamma) , \RV(h,\gamma)) .
\end{align}
\end{proposition}

\begin{proof}
Let $M \geq 0$ be given. Since $\rmD^+(g+h,x) \leq \rmD^+(g+h,x)$ and $\rmD^-(g+h,x) \geq \rmD^-(g,x)+\rmD^-(h,x)$ we have the following for any $x,y \geq M$ with $|x-y| \leq \gamma$:
\begin{align*}
    \frac{\rmD^+(g+h,x)}{\rmD^-(g+h,y)} 
    & \leq \frac{\rmD^+(g,x)+\rmD^+(h,x)}{\rmD^-(g,y)+\rmD^-(h,y)} \\
    & \leq \frac{\rmD^-(g,y)\RV(g,\gamma,M) + \rmD^-(h,y)\RV(h,\gamma,M)}{\rmD^-(g,y)+\rmD^-(h,y)} \\
    & \leq \max(\RV(g,\gamma,M),\RV(h,\gamma,M)).
\end{align*}
Taking the supremum over $x,y \geq M$ with $|x-y| \leq \gamma$ and then sending $M\to\infty$  gives \eqref{eq:BRV:M(g+h)2}.
\end{proof}

We also want to bound $\RV(gh,\gamma)$ for a pair of $\BRV$ functions $g$ and $h$, which turns out to be slightly more delicate, because (due to the Leibniz rule) we will need to control ratios of values of $g$ and $h$ as well as their derivatives. The following proposition will be helpful.
\begin{proposition} \label{prop:BRV:gRelBound}
 If $g \in \BRV$ and $\gamma > 0$, then
 \begin{equation} \label{eq:BRV:gRelBound}
\lim_{M \to \infty}\sup\set{\frac{g(x)}{g(y)} : x,y \geq M, |x-y| \leq \gamma} \leq \Lambda(g,2\gamma) \leq [\Lambda(g,\gamma)]^2.
 \end{equation}
\end{proposition}

\begin{proof}
Notice that the limit on the left-hand side of \eqref{eq:BRV:gRelBound} exists since the sets in question are decreasing in $M$. The second inequality in \eqref{eq:BRV:gRelBound} follows from  Proposition~\ref{prop:BRV:gamma12}, so it remains to to prove the first inequality in \eqref{eq:BRV:gRelBound}. Denoting $\Lambda = \Lambda(g,2\gamma)$, let us suppose for the sake of contradiction that there exists $\delta>0$ such that
\begin{equation} \label{eq:gRelbound:deltachoice}
\sup\set{\frac{g(x)}{g(y)} : x,y \geq M, |x-y| \leq \gamma} > \Lambda + \delta\end{equation}
for every $M$. Fix $0<\varepsilon<\delta$, and choose $k \in \bbN$ large enough that
\begin{equation} \label{eq:gRelboundKchoice}
    (\Lambda-1+\delta)\sum_{j=1}^k (\Lambda+\varepsilon)^{-j} > 1,
\end{equation}
which can clearly be done since the left-hand side of \eqref{eq:gRelboundKchoice}  converges to $$ \frac{\Lambda-1+\delta}{\Lambda-1+\varepsilon} >1. $$

Now, choose $M$ large enough that $M>k\gamma$ and
\begin{equation} \label{eq:gRelbound:epschoice} \Lambda(g,2\gamma,M-k\gamma) < \Lambda+\varepsilon.\end{equation} 
By \eqref{eq:gRelbound:deltachoice} and monotonicity of $g$, we may find $x \geq M$ for which $g(x+\gamma)/g(x) >\Lambda+\delta$. Naturally, this yields
\begin{equation}
    \frac{g(x+\gamma)-g(x)}{\gamma}
    > \frac{\Lambda-1+\delta}{\gamma}g(x) .
\end{equation}
Inductively applying \eqref{eq:gRelbound:epschoice} together with Lemma~\ref{lem:gratioToxRatioViaRV}, we observe that
\begin{equation}
    \frac{g(x-(j-1)\gamma ) - g(x-j\gamma)}{\gamma} 
    > (\Lambda+\varepsilon)^{-j} \frac{\Lambda - 1 + \delta}{\gamma} g(x)
\end{equation}
for all $j = 0,1,2,\ldots,k$. Thus,
\begin{align*}
    g(x-k\gamma)
    & = g(x) - \gamma\sum_{j=1}^k \frac{g(x-(j-1)\gamma) - g(x-j\gamma)}{\gamma} \\
    & < g(x) - \gamma \sum_{j=1}^k(\Lambda+\varepsilon)^{-j} \frac{\Lambda - 1+ \delta}{\gamma} g(x) \\
    & = g(x) \left( 1 - (\Lambda - 1 + \delta) \sum_{j=1}^k(\Lambda+\varepsilon)^{-j} \right) \\
    & < 0,
\end{align*}
which is a contradiction.
\end{proof}

\begin{proposition} \label{prop:BRV:gh}
Suppose $g,h:\bbR_+ \to \bbR_+$ are admissible. For all $\gamma > 0$, one has
\begin{align}
\label{eq:BRV:M(gh)2}
    \RV(gh,\gamma) &\leq \Lambda(g,\gamma)\Lambda(h,\gamma) \max[\Lambda(g,\gamma)\Lambda(h,\gamma)].
\end{align}
\end{proposition}

\begin{proof}First, note that
\begin{align}
\nonumber
    \rmD^+(gh,x) 
    & = \limsup_{z \to x} \frac{g(z)h(z)-g(x)h(x)}{z-x} \\
    \nonumber
    & \leq \limsup_{z \to x} \frac{g(z)h(z)-g(z)h(x)}{z-x} + \limsup_{z \to x} \frac{g(z)h(x)-g(x)h(x)}{z-x} \\
    \label{eq:Lambda+ghprodrule}
    & = g(x)\rmD^+(h,x)+h(x)\rmD^+(g,x).
\end{align}
Similarly,
\begin{equation} \label{eq:Lambda-ghprodrule}\rmD^-(gh,x) \geq g(x)\rmD^-(h,x) + h(x) \rmD^-(g,x).\end{equation}
Given $\varepsilon>0$, choose $M$ large enough that
\begin{equation} \label{eq:LambdaghMchoice1} 
\Lambda(g,\gamma,M) < \Lambda(g,\gamma)+\varepsilon, \quad \Lambda(h,\gamma,M) < \Lambda(h,\gamma)+\varepsilon
\end{equation}
and use Proposition~\ref{prop:BRV:gRelBound} to ensure that $M$ is also large enough that
\begin{equation}\label{eq:LambdaghMchoice2} 
    g(x) \leq g(y)[\Lambda(g,\gamma)^2 + \varepsilon], \quad 
    h(x) \leq h(y)[\Lambda(h,\gamma)^2 + \varepsilon], \quad
    \forall x,y \geq M, \ |x-y| \leq \gamma.
\end{equation}
 
Let $x, y \geq M$ with $|x-y| \leq \gamma$ be given. Putting together \eqref{eq:Lambda+ghprodrule}, \eqref{eq:Lambda-ghprodrule}, \eqref{eq:LambdaghMchoice1}, and \eqref{eq:LambdaghMchoice2}
\begin{align*}
      & \, \frac{\rmD^+(gh,x)}{\rmD^-(gh,y)} \\
      \leq & \, \frac{g(x)\rmD^+(h,x) + h(x)\rmD^+(g,x)}{g(y)\rmD^-(h,y) + h(y)\rmD^-(g,y)} \\
     \leq & \,\frac{[\Lambda(g,\gamma)^2+\varepsilon](\Lambda(h,\gamma)+\varepsilon)g(y)\rmD^-(h,y) + (\Lambda(g,\gamma)+\varepsilon)[\Lambda(h,\gamma)^2 + \varepsilon] h(y)\rmD^-(g,y)}{g(y)\rmD^-(h,y) + h(y)\rmD^-(g,y)} \\
     \leq & \, \max\left([\Lambda(g,\gamma)^2+\varepsilon](\Lambda(h,\gamma)+\varepsilon) ,  (\Lambda(g,\gamma)+\varepsilon)[\Lambda(h,\gamma)^2 + \varepsilon] \right).
\end{align*}
Sending $M \to \infty$ and $\varepsilon\downarrow 0$ gives the desired result.
\end{proof}

\begin{proposition}

The sets $\TRV$ and $\BRV$ are closed under finite sums, products, and scaling by positive constants. 
\end{proposition}

\begin{proof}
This follows immediately from the bounds in Propositions~\ref{prop:BRV:g+h} and \ref{prop:BRV:gh}.
\end{proof}

Let us now give some examples of functions with bounded relative variation.

\begin{proposition}\mbox{\,}
 \begin{enumerate}[{\rm(a)}]
     \item For any $m >0$, $g(x)=x^m$ is an admissible $\TRV$ function.
     \item For any $a,b > 0$, $g(x) = e^{ax^b}$ is in $\BRV$ if and only if $b \leq 1$ and in $\TRV$ if and only if $b <1$
 \end{enumerate}
\end{proposition}

\begin{proof}
Since these functions are differentiable, $\rmD^+(g,x) = \rmD^-(g,x) = g'(x)$, which we use throughout the proof.
\medskip

(a) Denote $g(x) = x^m$. If $x,y \geq M$, $|x-y|\leq \gamma$, and $m \geq 1$ then
\begin{align*}
    \frac{\rmD^+(g,x)}{\rmD^-(g,y)}
     = \frac{g'(x)}{g'(y)} 
     = \frac{x^{m-1}}{y^{m-1}} 
     \leq \left( \frac{y+\gamma}{y}\right)^{m-1} 
     \leq \left(1+\gamma M^{-1}\right)^{m-1},
\end{align*}
which converges to one as $M \to \infty$. A similar argument works when $0<m<1$, but one must bound things differently since $t\mapsto t^{m-1}$ is decreasing for $m<1$.
\medskip

(b) Notice that $g'(x) = abx^{b-1}e^{ax^b}$. If $x,y \geq M$, $|x-y| \leq \gamma$, and $0 < b \leq 1$, we get
\begin{align*}
    \frac{\rmD^+(g,x)}{\rmD^-(g,y)}
     = \frac{g'(x)}{g'(y)} 
     = \frac{x^{b-1}}{y^{b-1}}e^{a(x^b-y^b)} 
     \leq (1-\gamma M^{-1})^{b-1}e^{a((y+\gamma)^b-y^b)}.
\end{align*}
The right-hand side converges to one as $M \to \infty$ if $b<1$ and converges to a finite value if $b=1$. On the other hand, if $b=1$, then one can check that
\begin{align*}
    \frac{g'(x+\gamma)}{g'(x)} = 
     e^{a\gamma} > 1,
\end{align*}
for all $x$. This shows that $g \in \BRV \setminus \TRV$. Similar calculations show that $g \notin \BRV$ whenever $b>1$.
\end{proof}

Let us briefly note that every $\BRV$ function is exponentially bounded:

\begin{corollary} \label{coro:BRVexpbounded}
If $g \in \BRV$, then $g$ is exponentially bounded, that is, $g(x) \leq Ae^{Bx}$ for constants $A,B>0$.
\end{corollary}

\begin{proof}
Proposition~\ref{prop:BRV:gRelBound} implies that for some $\varepsilon > 0$, some large $x$ and all $n \in \bbN$,
\begin{equation}
    g(x + n\gamma) \leq g(x)(\Lambda(g,\gamma)^2+\varepsilon)^n.
\end{equation}
The result follows by monotonicity.
\end{proof}

The converse of Corollary~\ref{coro:BRVexpbounded} fails: for any increasing function $h$, one can find an admissible function in the complement of $\BRV$ that is dominated by $h$.

\begin{example}
For any continuous increasing function $h:\bbR_+ \to \bbR_+$ such that $h(x)\to\infty $ as $x \to \infty$, there is an admissible function $g$ such that $g(x) \leq h(x)$ for all sufficiently large $x$ and $g \notin \BRV$.

To see this, 
choose $0 = x_0 < x_1 < \cdots$ so that $h(x_n) = n$ (and hence $h(x) \geq n$ for $x \geq x_n$). Pick $0=y_0 < y_1 < \cdots$ such that
\begin{align}
\label{eq:slowgrowNotBRV1}    y_n & \geq x_n \\
\label{eq:slowgrowNotBRV2}    y_{n+1}-y_n & \geq n(y_n - y_{n-1}) \quad \forall n \in \bbN.
\end{align}
Define $g$ to be continuous and piecewise affine with $g(y_n)=n/2$ for each $n \in \bbZ_+$. Notice that $g$ is admissible, that the definition of $g$ and \eqref{eq:slowgrowNotBRV1}  ensure $g(x) \leq h(x)$ for $x \geq x_1$, an that \eqref{eq:slowgrowNotBRV2} ensures that $g \notin \BRV$.

One can mollify this example to produce a smooth admissible $g \in \BRV$ having similar properties.
\end{example}

\section{Proofs of Main Theorems} \label{sec:BSC}

Let us prove Theorems~\ref{t:AbstracBSC} and \ref{t:AbstracBSCgen}. Clearly the latter implies the former, so we focus on proving Theorem~\ref{t:AbstracBSCgen}. We follow the strategy from \cite{DFG2021JFA}.

\begin{lemma} \label{lem:thicknessUnderTransformation}
Suppose $g \in \BRV$, $K \subseteq [M,\infty)$ is compact, and $\diam(K) \leq \gamma$. Then,
\begin{equation}
    \tau(g[K]) \geq [\Lambda(g,\gamma,M)]^{-1}\tau(K).
\end{equation}
\end{lemma}

\begin{proof}
If $\tau(K) = \infty$, then $K$ and $g[K]$ are both intervals, and there is nothing to do, so assume $\tau:=\tau(K)<\infty$, and write $I = [\min K, \max K]$. Given $\varepsilon >0$, choose a presentation $\mathcal{U} = \{U_n\}$ such that $\tau(K,\mathcal{U}) > \tau - \varepsilon$. 

For any two intervals $B=[u,v]$, $U=[v, w]$ from $I$ such that
$$
\frac{|B|}{|U|}=\frac{v-u}{w-v} \ge \tau - \varepsilon,
$$
we have by Lemma~\ref{lem:gratioToxRatioViaRV}
\begin{align*}
\frac{|g(B)|}{|g(U)|} 
 = \frac{g(v)-g(u)}{g(w)-g(v)} 
  \geq [\RV(g,\gamma,M)]^{-1}\frac{v-u}{w-v}
 \geq [\RV(g,\gamma, M)]^{-1}(\tau - \varepsilon).
\end{align*}
 A similar estimate holds for intervals situated in the other order, that is, $B=[u,v]$, $U=[w, u]$. Therefore, $g\mathcal{U} = \{g[U_n]\}$ is a presentation of $g[K]$ satisfying $\tau(g[K],g\mathcal{U}) \geq  [\RV(g,\gamma, M)]^{-1}(\tau - \varepsilon)$. Thus,
 \begin{equation}
     \tau(g[K]) \geq  [\RV(g,\gamma, M)]^{-1}(\tau - \varepsilon).
 \end{equation}
 Since this holds for arbitary $\varepsilon>0$, the lemma follows.
\end{proof}

\begin{lemma} \label{lem:thicknessLongestgap}
If $K \subseteq \bbR$ is compact and $\tau(K) > \beta$, then the longest gap length of $K$ satisfies
\begin{equation}
    \gamma(K) \leq \frac{\diam(K)}{1+2\beta}.
\end{equation}
\end{lemma}

\begin{proof}
Write $I$ for the convex hull of $K$ so that $\diam(K) = |I|$. If the longest gap $U$ of $K$ satisfies $|U| > \frac{|I|}{1+2\beta}$, $I \setminus U$ has two components, one of which must have length no larger than $$\frac{|I|-|U|}{2} < \frac{1- \frac{1}{1+2\beta}}{2}|I| = \frac{\beta}{1+2\beta} |I|.$$
Thus, for (at least) one endpoint $u$ of $U$, one has
\[\tau(K,\mathcal{U},u) \leq \frac{\frac{\beta}{1+2\beta} |I|}{\frac{1}{1+2\beta}|I|} = \beta\]
for every presentation $\mathcal{U}$ of $K$, leading to $\tau(K) \leq \beta$. The result follows.
\end{proof}

\begin{proof}[Proof of Theorem~\ref{t:AbstracBSCgen}]
Assume $g \in \BRV$, $F\subseteq \bbR_+$ is $(A,a,1+\varepsilon)$-thick, and let $\{K_n\}_{n=0}^\infty$ be an ordered fragmentation of $F$ satisfying \eqref{eq:AatThick1}, \eqref{eq:AatThick2}, and \eqref{eq:AatThick3}. Denote $\widetilde K_n=g(K_n)$, let $I_n$ denote the convex hull of $K_n$, and assume 
\begin{equation} \label{eq:BSCGenLambdachoice}
    \Lambda(g,A) < \min\left\{ \sqrt{\frac{1+\varepsilon}{1+\frac{\varepsilon}{2}}}, \sqrt[5]{\frac{3}{2}}, \sqrt[3]{\frac{A}{a}} \right\}.
\end{equation}
We may also choose $M$ large enough that $\Lambda(g,A,M)$ is strictly less than the right-hand side of \eqref{eq:BSCGenLambdachoice} as well.

By Lemma~\ref{lem:thicknessUnderTransformation} and \eqref{eq:BSCGenLambdachoice}, if $n$ is large enough that $K_n \subseteq [M,\infty)$, one has
\begin{align*} 
\tau(\widetilde K_n)\ge [\Lambda(g,2A,M)]^{-1}(1+\varepsilon)
 & \geq [\Lambda(g,A,M)]^{-2}(1+\varepsilon) \\
 & > \left[ \frac{1+\varepsilon}{1+\frac{\varepsilon}{2}} \right]^{-1}(1+\varepsilon) \\
 & = 1+\frac{\varepsilon}{2},
\end{align*}
and thus $\tau(\widetilde K_n) > 1+\frac{\varepsilon}{2}$ for all sufficiently large $n$. Consequently, the sum $\widetilde K_n+\widetilde K_{n}$ is an interval by Remark~\ref{r.cplusc}.

Next, let us show that for all sufficiently large $n$, each of the sets $\widetilde K_n$ and  $\widetilde K_{n+1}$ has diameter larger than the largest gap of the other. In that case Lemma~\ref{l.sumofcs} will imply that $\widetilde K_n+\widetilde K_{n+1}$ is an interval. Write $I_n=[x_n,y_n]$ for each $n$. By our assumptions, we have
\begin{equation} \label{eq:ABSC:basicIneqs}
A \le y_n-x_n \le 2A, \quad x_{n+1}-y_n \le a
\end{equation}
for every $n$.
Since we already know that $\tau(\widetilde K_{n+1})>1$, Lemma~\ref{lem:thicknessLongestgap} implies that the largest gap of $\widetilde K_{n+1}$ is not greater than $\frac{1}{3}(g(y_{n+1})-g(x_{n+1}))$, and the diameter of $\widetilde K_n$ is equal to $g(y_n)-g(x_n)$. Now observe
\begin{align*}
    \frac{g(y_n) - g(x_n)}{\frac{1}{3}[g(y_{n+1}) - g(x_{n+1})]}
    & \geq [\RV(g,4A+a,M)]^{-1} \frac{y_n-x_n}{\frac{1}{3}(y_{n+1}-x_{n+1})} \\
    & \geq [\RV(g,A,M)]^{-5} \frac{3A}{2A} \\
    & > \left[ \frac{3}{2} \right]^{-1} \frac{3}{2} \\
    & = 1,
\end{align*}
by \eqref{eq:BSCGenLambdachoice}. Consequently, for all sufficiently large values of $n$, the diameter of $\widetilde K_n$ is greater than the largest gap of $\widetilde K_{n+1}$. Similarly one can show that for all sufficiently large values of $n$, the diameter of $\widetilde K_{n+1}$ is greater than the largest gap of $\widetilde K_{n}$.

Consequently, the sets $J_n := \widetilde K_n + \widetilde K_{n}$ and $J_n' := \widetilde K_n + \widetilde K_{n+1}$ are intervals for large $n$. Let us show that they cover a half line. To conclude, it suffices to verify $J_n \cap J_n'\neq \emptyset$ and $J_n' \cap J_{n+1} \neq \emptyset$.

Recall $I_n = [x_n, y_n]$. It follows from our discussion above that
\begin{align*}
    J_n & = [2g(x_n), 2g(y_n)],\\
    J_{n+1} & = [2g(x_{n+1}), 2g(y_{n+1})],\\
    J_n'    & = [g(x_n)+g(x_{n+1}), g(y_n) + g(y_{n+1})].
\end{align*}
    To show that $J_n$ is not disjoint from $J_n'$ we need to check that $2g(y_n) \geq g(x_n)+g(x_{n+1})$. To that end, note that
    \begin{align*}
        \frac{g(y_n) - g(x_n)}{g(x_{n+1})-g(y_n)}
        & \geq [\RV(g,2A+a,M)]^{-1} \frac{y_n-x_n}{x_{n+1}-y_n} \\
        & \geq [\RV(g,A,M)]^{-3} \frac{A}{a} \\ 
        & > 1,
    \end{align*}
again by \eqref{eq:BSCGenLambdachoice}.

One can show that $J_n'$ is not disjoint from $J_{n+1}$ from an almost identical argument.

Putting everything together, the set
$$
\bigcup_{n \ge n_0} \left( J_n \cup J_n' \right),
$$
is a half line for large enough $n_0$. Since this set is contained in $\tilF + \tilF$, we are done.
\end{proof}

\begin{proof}[Proof of Theorem~\ref{t:AbstracBSC}]
This follows from Theorem~\ref{t:AbstracBSCgen}.
\end{proof}

Let us comment on the assumptions and why they are necessary. Consider first any semibounded closed set $F\subseteq \bbR_+$ with ordered fragmentation $\{K_n\}$. One can clearly choose a smooth, nondecreasing function $f \in C^\infty(\bbR_+)$ that satisfies $f|_{K_n} \equiv n$ for each $n \geq 0$. Clearly then $$\tilF + \tilF = \bbZ_+$$ 
which certainly does not contain a half-line. Of course, this $f$ is clearly not admissible. However, one can certainly perturb about this situation somewhat. Concretely, one can choose $f:\bbR_+ \to \bbR_+$ smooth and increasing with $\tilK_n = f[K_n] \subseteq [n-\varepsilon,n+\varepsilon]$ for $n \in \bbZ_+$. One still sees that $\tilF+\tilF$ is contained in the $2\varepsilon$-neighborhood of $\bbZ_+$. Evidently, the mechanism that produces this is that $f'\sim \varepsilon /A$ on the convex hull of $K_n$ while $f' \sim 1/a$ in between successive $K_n$'s, leading to relative variations of $f'(x) = \rmD^\pm(f,x)$ on the order of $\frac{A}{\varepsilon a}$.
\bigskip

Let us conclude the present section with the proof of Theorem~\ref{t:BSC:bigtau}. Since this is similar to that of Theorem~\ref{t:AbstracBSCgen}, we will only give the main steps.

\begin{proof}[Proof of Theorem~\ref{t:BSC:bigtau}]
Choose $a_0, \tau_0 >0$ such that
\begin{equation} \label{eq:BSCalt:a0tau0choice}
    \tau_0 > R^{7}, 
    \quad a_0 < A/R^3.
\end{equation}
Since $R > 1$, note additionally that $a_0 < A$. Now, assume $\Lambda(g,A) \leq R$ and that $F$ is $(A,a,\tau)$-thick with $0<a\leq a_0$ and $\tau \geq \tau_0$, let $\{K_n\}$ denote an ordered fragmentation of $F$ satisfying \eqref{eq:AatThick1}, \eqref{eq:AatThick2}, and \eqref{eq:AatThick3}, and use the same notation as in the proof of Theorem~\ref{t:AbstracBSCgen}. Using Proposition~\ref{prop:BRV:gamma12}, we note
\begin{equation}
    \Lambda(g,2A) < R^2,
\end{equation}
so following the steps at the beginning of the proof of Theorem~\ref{t:AbstracBSCgen}, we have $\tau(\widetilde K_n) \geq \tau_0/R^2 > R^5 > 1$ for large enough $n$. 

For large enough $n$ that the previous thickness statement holds true, the largest gap of $\widetilde{K}_{n+1}$ is smaller than $\frac{g(y_{n+1})-g(x_{n+1})}{2R^5+1}$ by Lemma~\ref{lem:thicknessLongestgap}, and we have
\begin{align*}
    \frac{g(y_n) - g(x_n)}{\frac{1}{2R^5+1}(g(y_{n+1}) - g(x_{n+1})}
    & \geq R^{-5} \frac{A}{\frac{1}{2R^5+1}2A} \\ 
    & > 1,
\end{align*}
showing that the diameter of $\widetilde K_n$ exceeds the size of the largest gap of $\widetilde{K}_{n+1}$ for large enough $n$ (and vice versa by the same argument). The assumption on $a_0$ ensures that 
$$\Lambda(g,2A+a) \leq \Lambda(g,3A) \leq R^3 < A/a$$
for large enough $n$. These ingredients suffice to apply the arguments of the previous proof and conclude that $\tilF + \tilF$ contains a half-line.
\end{proof}

\section{Examples Not Containing Half-Lines}
\label{sec:nohalfline}

We now turn towards the construction of suitable examples whose sums do not contain half-lines when the assumptions of the main theorems are not met.

\begin{lemma} \label{lem:counterCriterion}
Suppose
\begin{equation}
    F  \subseteq \bigcup_{n=0}^\infty [x_n,y_n]
\end{equation}
where $x_n < y_n < x_{n+1}$  for every $n$, and suppose that for some $N_0$, one has
\begin{equation} \label{eq:counterxample:assmpt1}
    2y_n - x_0 < x_{n+1} \quad \forall n \geq N_0.
\end{equation}
Then $F+F$ does not contain a half-line.
\end{lemma}

One can generalize this to sums consisting of more than two sets, which may be distinct. We shall do that presently and derive Lemma~\ref{lem:counterCriterion} as a consequence of a more general statement.

\begin{definition}
For $F_1,\ldots,F_d \subseteq \bbR$, define
\begin{equation}
    \sum_{j=1}^d F_j = F_1 + \cdots + F_d
    = \left\{ \sum_{j=1}^d a_j : a_j \in F_j \ \forall 1 \le j \le d \right\}.
\end{equation}
\end{definition}

\begin{lemma} \label{lem:counterCriterionGen}
Suppose for $j=1,2,\ldots,d$
\begin{equation}
    F_j \subseteq \bigcup_{n=0}^\infty [x_{n,j},y_{n,j}]
\end{equation}
is contained in a union of closed, bounded intervals such that $x_{n,j} < y_{n,j} < x_{n+1,j}$ for all $n \geq 0$ and $j=1,2,\ldots,d$. If for some $N_0$, one has
\begin{equation}
    \sum_{j=1}^d y_{n,j} < \min_{k=1,2,\ldots,d}\left(x_{n+1,k} + \sum_{j\neq k}x_{0,j} \right) \quad \forall n \geq N_0.
\end{equation}
Then $\sum_{j=1}^d F_j$ does not contain a half-line.
\end{lemma}

\begin{proof}
Denote $F = \sum_{j=1}^d F_j$, and define $I_{n,j} = [x_{n,j},y_{n,j}]$. For $k =1,2,\ldots,d$ and $n \geq 0$, define the $(n,k)$-\emph{stratum} by
\[S_{n,k} = \bigcup_{\substack{0 \le n_1,n_2,\ldots,n_d \leq n \\ n_k = n}} \sum_{j=1}^d I_{n_j,j}.\]
For $n \geq 0$, define 
\[ S_n = \bigcup_{k=1}^d S_{n,k},
\quad T_n^- = \bigcup_{m=0}^n S_m, \quad T_n^+ = \bigcup_{m=n+1}^\infty S_m.  \]
and note that
\[ F \subseteq T_n^- \cup T_n^+ \quad \forall n \in \bbN.\]
Since \[\max(T_n^-)= \sum_{j=1}^d y_{n,j} \text{ and } \min(T_n^+) = \min_{k=1,2,\ldots,d}\left(x_{n+1,k} + \sum_{j\neq k}x_{0,j} \right),\]
our assumption \eqref{eq:counterxample:assmpt1} yields
\begin{equation}
    \emptyset \neq G_n := (\max T_n^-, \min T_n^+) \subseteq \bbR\setminus F\quad \forall n \geq N_0,
    \end{equation}
which suffices to show that $F$ does not contain a half-line.
\end{proof}

\begin{proof}[Proof of Lemma~\ref{lem:counterCriterion}]
This follows immediately from Lemma~\ref{lem:counterCriterionGen}.
\end{proof}

\begin{proof}[Proof of Theorem~\ref{t:main:counterex}]
Let $F$ be $a$-sparse, and write
\[F = \bigcup_{n=0}^\infty K_n\]
for an ordered fragmentation with $\dist(K_n,K_{n+1}) \geq a$. Writing $[x_n,y_n]$ for the convex hull of $K_n$, note that
\[\tilF \subseteq \bigcup_{n=0}^\infty [\widetilde x_n, \widetilde y_n]\]
with $\widetilde x_n = e^{rx_n}$, $\widetilde y_n =e^{ry_n}$. By assumption $ra\geq \log 2$, so we observe
\begin{align*}
    2\widetilde y_n -  \widetilde x_0
    & = 2e^{ry_n} - e^{rx_0} \\
    & \leq e^{r(y_n+a)} - e^{rx_0} \\
    & \leq e^{rx_{n+1}} -e^{rx_0} \\
    & =\widetilde x_{n+1} - e^{rx_0} \\
    & < \widetilde{x}_{n+1},
\end{align*}
in which the third line follows from sparsity. Thus the claim follows from Lemma~\ref{lem:counterCriterion}.
\end{proof}

\begin{definition}
Given constants $A,a>0$, let $F(A,a)$ denote a union of intervals of length $A$ separated by a uniform distance of $a$ between consecutive intervals, that is,
\[ F(A,a) = \bigcup_{n=0}^\infty [n(A+a),n(A+a)+A].  \]
\end{definition}

This can be used to show that the bound $ra \geq \log 2$ is sharp for constructing counterexamples that do not contain a half-line.

\begin{proposition} \label{prop:erxPhaseTransition}
 Let $A,a,r>0$ and $d \geq 2$ be given, and consider $g(x)=e^{rx}$ and $\tilF = g[F(A,a)]$. 
 \begin{enumerate}[{\rm(a)}]
\item If
 \begin{equation}
     ra \geq \log(2),
 \end{equation}
 then $\tilF+\tilF$ does not contain a half-line.
 \item If
 \begin{equation}
     ra<\log(2)
 \end{equation}
 and $A$ is sufficiently large, then $\tilF+\tilF$ contains a half-line.
 \end{enumerate}
\end{proposition}

\begin{proof}
(a)  Assume  that $ra \geq \log(2)$. Since $F(A,a)$ is clearly $a$-sparse, this follows from Theorem~\ref{t:main:counterex}.
\medskip

(b) On the other hand, suppose $ra<\log(2)$ and choose $A$ large enough that
\begin{equation} \label{eq:phasetransition:Alargecond}
    e^{-rA}+e^{ra}<2, \qquad e^{-ra}+e^{rA} > 2.
\end{equation} 
Define $x_n = n(A+a)$, $y_n = n(A+a)+A$, $\widetilde x_n = e^{rx_n}$, and $\widetilde y_n = e^{ry_n}$ so that \[F = \bigcup_{n=0}^\infty [x_n,y_n], \quad \tilF = \bigcup_{n=0}^\infty [\widetilde x_n, \widetilde y_n]\]
Observe that $\tilF+\tilF$ contains the intervals $J_n = [2\widetilde x_n,2 \widetilde y_n]$ and $J_n'=[\widetilde x_n + \widetilde x_{n+1}, \widetilde y_n + \widetilde y_{n+1}]$. Observe that \eqref{eq:phasetransition:Alargecond} yields
\begin{align}
\nonumber
    2\widetilde y_n
    & = 2e^{rn(A+a)+rA} \\
    \nonumber
    & > (e^{-rA}+e^{ra})e^{rn(A+a)+rA} \\
    \nonumber
    & = e^{rn(A+a)} + e^{r(n+1)(A+a)} \\
    \label{eq:phasetransition:halfline1}
    & =\widetilde x_n + \widetilde x_{n+1}.
\end{align}
Similarly, \eqref{eq:phasetransition:Alargecond} gives
\begin{align}
\nonumber
\widetilde y_{n+1} + \widetilde y_n
& = e^{r(n+1)(A+a)+rA}+e^{rn(A+a)+rA} \\
\nonumber
&  = e^{r(n+1)(A+a)}(e^{rA}+e^{-ra}) \\
\nonumber
& > 2e^{r(n+1)(A+a)} \\
\label{eq:phasetransition:halfline2}
& = 2 \widetilde x_{n+1}.
\end{align}
Thus, $\tilF +\tilF$ contains
\[ \bigcup_n J_n \cup J_n', \]
which contains a half-line in view of \eqref{eq:phasetransition:halfline1} and \eqref{eq:phasetransition:halfline2}. 
\end{proof}

\bibliographystyle{abbrv} 
\bibliography{FTbib}

\end{document}